\title{Permutation-based multiple testing when fitting many generalized linear models}

\documentclass[12pt]{article}

\usepackage{amsmath}
\usepackage{amsfonts}
\usepackage{authblk}
\usepackage{bm}
\usepackage{bbm}
\usepackage{amsthm}
\usepackage{comment}
\usepackage[nottoc]{tocbibind}
\usepackage{graphicx}
\usepackage{subcaption}
\usepackage{changepage}
\usepackage{natbib}
\usepackage{abstract}
\usepackage[colorlinks=true,citecolor=blue,runcolor=blue,linkcolor=blue, pdfborder={0 0 0}]{hyperref}

\usepackage{xcolor}

\usepackage{tikz}
\usepackage{appendix}

\definecolor{cobalt}{rgb}{0.0, 0.28, 0.67}
\definecolor{darkblue}{rgb}{0.0, 0.0, 0.55}
\definecolor{dgreen}{rgb}{0.0, 0.4, 0.0}

\theoremstyle{plain}
\newtheorem{theorem}{Theorem}

\newtheorem{lemma}[theorem]{Lemma}

\theoremstyle{definition}

\newcommand{\fgamma}{\bm{\gamma}}

\newcommand{\black}{\color{\black}}
\newcommand{\N}{\mathcal{N}}

\author[1]{Riccardo De Santis\footnote{To contact:  \textit{riccardo.desantis2@unisi.it}}}
\author[2]{Jelle J. Goeman}
\author[3]{Samuel Davenport}
\author[4]{Jesse Hemerik}
\author[5]{Livio Finos}

\affil[1]{University of Siena, Italy} 
\affil[2] {Leiden University Medical Center, The Netherlands} 
\affil[3]{Division of Biostatistics, University of California San Diego, United States} 
\affil[4]{Erasmus University Rotterdam, The Netherlands}
\affil[5]{University of Padova, Italy} 
\date{}

\begin{document}
\maketitle

\begin{abstract}
    In many applied sciences a popular analysis strategy for high-dimensional data is to fit many multivariate generalized linear models in parallel. This paper presents a novel approach to address the resulting multiple testing problem by combining a recently developed sign-flip test with permutation-based multiple-testing procedures. Our method builds upon the univariate standardized flip-scores test which offers robustness against misspecified variances in generalized linear models, a crucial feature in high-dimensional settings where comprehensive model validation is particularly challenging. We extend this approach to the multivariate setting, enabling adaptation to unknown response correlation structures. This approach yields relevant power improvements over conventional multiple testing methods when correlation is present.
\end{abstract}

\section{Introduction}
In high-dimensional data, such as neuroimaging and transcriptomics, it is common to fit many generalized linear regression models in parallel, each with a relatively small sample size \citep{schaarsmidt2022multiglm, RnaSeq2014Love, winkler2014permutation, davenport2022fdp}. Usually, the goal of the analysis is to perform hypothesis testing to find relevant associations, typically after adjusting the $p$-values for multiple testing.

This approach comes with several challenges. In the first place, small sample sizes can make classical tests unreliable, especially in nonlinear models where exact methods are not available and tests rely on asymptotic arguments. In such cases, the usual normal approximation of the test statistic can be quite unreliable \citep{schaarsmidt2022multiglm}, especially in the extreme tails that are relevant in the context of multiple testing. 
Secondly, the generalized linear model makes some crucial assumptions which are difficult to check, especially for nonlinear models with small sample size. In particular, the detection of overdispersion among the assumed models can be quite problematic, a problem which is often exaccerbated as the variance is not constant among the observations. Proper model checking in high-dimensional data is further hampered by the sheer number of models that are fit. Finally, the test statistics of the parallel models are often correlated, due to correlations in the underlying biological measurements. Classical multiple testing corrections (such as Bonferroni-Holm) are designed in order to protect against any correlation structure, but can be very conservative in the presence of such correlations  \citep{gao.etal2010permut, goeman2014tutorial, saffari.etal2018permut}. Instead, as we shall see, taking the correlation into account allows for increases in power.

Earlier work in this direction have been addressed by \cite{hothorn2008simultaneous, pipper2012versatile, bretz2016multiple, pallmann2018simultaneous} which
tackle the multiple testing issue -- in contexts that comprise GLMs -- by estimating the correlation structure between the test stiatistics and the inference is based on a parametric approach. Unfortunately, they do not allow to deal cases of very high dimensionality as the ones we are tacking in this work. 

Some solution is offered by resampling-based methods in the restricted field of linear models. \cite{winkler2014permutation} provides a thoughtful review of these methods. The asymptotic control of the type I error is good in most of the practical scenarios; but they are not robust to heteroscedastic errors (i.e. misspecified variances)  \citep{https://doi.org/10.1111/j.1442-9993.2001.01070.pp.x,mcardle2001fitting} or they are robust in practice, but a formal proof is not provided.
The broader field of generalized linear regression models with high dimensional responses offer even less satisfactory tools. As an example, \cite{10.1093/bioinformatics/btn650} propose a solution where the effect of the confounders are accounted by matching units with similar euclidean distance in the space of the confounders. Despite being an effective approach, the solution is ad hoc since the quality of the control of the false positives depends on the correct choice of parameters.

In this work we propose a permutation-based multiple testing procedure in combination with the sign-flip score test \citep{hemerik2020robust, de2022inference}.
The test shows reliable behavior for small sample sizes, and has been shown to be robust against general variance misspecification under minimal assumptions. We start by defining a test for a global null hypothesis about a multivariate regression parameter, which guarantees weak control of the familywise error rate (FWER). We extend this to allow for additional inference by embedding the global test in a closed testing procedure \citep{marcus1976closed}.
This allows us to compute adjusted $p$-values for each individual hypothesis through a multiple testing procedure based on the max-$T$ method of \cite{westfall1993resampling}, which is a dramatic shortcut of the closed testing procedure. In particular this approach guarantees strong control of the FWER, that is, it produces valid adjusted $p$-values for all possible subsets of hypotheses. 

Crucially our approaches provides inference by jointly sign-flipping the score contributions allowing the method to adapt to the unknown correlation structure. This is especially useful when strong correlation between the individual test statistics is present; providing a relevant gain in power over alternative methods, as we will show in a simulation. 

The paper is organized as follows: Section \ref{Sect:univartest} revisits the sign-flip score test for univariate testing; Section \ref{sec:multiv} contains the novel contribution of the paper, that is, the multivariate testing extension. Section \ref{sect:sims} contains a simulation study; Section \ref{sec:concl} contains the conclusions.

Code to implement our methods is available in the flipscores R package \citep{Rflipscores}, available from CRAN. Python and Matlab implementations are also available \citep{pyperm, matperm}. Code to reproduce the results of this paper is available at \url{github.com/livioivil/jointest}.

\section{One dependent variable} \label{Sect:univartest}
In this section we review the derivation of the univariate sign-flip score test as detailed in \citet{hemerik2020robust} and \citet{de2022inference} before we extend to the multi-variable setting in Section \ref{sec:multiv}.

Given $n$ specify a number of units, which could e.g. correspond to the number of subjects in a given analysis. Then for every $1 \leq i \leq n$, let $Y_i$ be the target variable which is assumed to belong to the exponential dispersion family, i.e., with a density of the form \citep{agresti2015foundations}
\begin{equation*} \label{eq:model}
    f(y_i;\theta_i,\phi_i)=
    \exp\left\{\frac{y_i\theta_i-b(\theta_i)}{a(\phi_i)}+c(y_i,\phi_i)\right\},
\end{equation*}
where $\theta_i$ and $\phi_i$ are respectively the canonical and the dispersion parameter while $a(\cdot)$, $b(\cdot)$ and $c(\cdot)$ are known functions. Consequently, the mean and variance of the observed outcome are defined as
\begin{equation*}
\mu_i=E(Y_i)=b'(\theta_i); \qquad
Var(Y_i)=b''(\theta_i)a(\phi_i).
\end{equation*} 

The expected value of the vector $Y=(Y_1,\dots,Y_n)^T$ is assumed to depend on some covariates through the relation
$$\mathbb{E}(Y)=g^{-1}(X\beta+Z\gamma)$$
where $g(\cdot)$ is the link function, the covariate $X$ is an $n\times 1$ matrix, i.e., a column vector, such that $\beta\in\mathbb{R}$. Further, the nuisance covariates $Z$ are a $n\times (k-1)$ matrix. Let $\eta=X\beta+Z\gamma$ define the linear predictor.

Throughout this section we are interested in testing the null hypothesis $H_0:\beta=\beta_0$ against a one or two-sided alternative. The other parameters $\gamma$ (and $\phi_i)$ are thus considered as nuisance parameters. Note that in practice in order to fit generalized linear models it is common to assume that $\phi_i=1$ or $\phi_i=\phi$, a requirement needed to be able to estimate them \citep{agresti2015foundations}. This turns out to an assumption about the variance structure that can be quite restrictive. We do not wish to have this kind of assumption, and so instead rely on score based sign-flipping tests which we introduced in \cite{hemerik2020robust} and \cite{de2022inference} which have robustness against variance misspecification. To do so let
$$D=\mathrm{diag}\left\{\frac{\partial \mu_i}{\partial \eta_i}\right\}; \qquad V=\mathrm{diag}\{\mathrm{Var}(y_i)\}.$$
The diagonal matrix of the GLM weights is then defined as $W=DV^{-1}D$ with entries $w_i= ({\partial \mu_i}/{\partial \eta_i})^2/var(Y_i)$ \citep{agresti2015foundations}.
Finally,
$$P= {W}^{1/2}Z(Z'{W}Z)^{-1}Z'{W}^{1/2}$$
represents the projection matrix for GLMs. Note that $D, V$ and $P$ are functions of the regression parameters $\beta, \gamma$, and we evaluate them under the null hypothesis $H_0$ i.e. with $\beta = \beta_0$ and $\gamma = \gamma_0$.

We will consider a test based on the \emph{effective score} which, when the estimated nuisance parameter $\hat{\gamma}$ is plugged in, is defined as 
\begin{equation*} n^{1/2}S^{*}_{\hat{\gamma}}=X' {W}^{1/2} (I-{P}){V}^{-1/2}(Y-\hat{\mu})=\sum_{i=1}^{n}\nu^{*}_{\hat{\gamma},i}.
\end{equation*}
The effective score can be interpreted as the residual from the projection of the marginal score for $\beta$ onto the space spanned by the nuisance scores \citep{marohn2002comment}. Importantly the effective score can be written as a sum of $n$ elements $\nu^{*}_{\hat{\gamma},i}$, which we call the score contributions. 

In what follows we shall adopt the same model assumptions and estimation strategy as in \cite{de2022inference}. In particular, the link function is crucially assumed to be correctly specified which implies that we are able to consistently estimate the regression parameter $\gamma$.

The test is performed by means of random sign flipping of the score contributions. Define $g_1 =(1, \dots, 1)\in \mathbb{R}^n$ and let $g_2,...,g_w$, for some number of flips $w \in \mathbb{N}$,  be sampled with replacement from  $\{-1, 1\}^n$, unless stated otherwise. The sign-flipped effective score statistics are then defined as
\begin{align*}
S^{*j}_{\hat{\fgamma}}&=n^{-1/2} \sum_{i=1}^{n}   g_{ji}\nu^{*}_{\hat{\fgamma},i},
\end{align*}
where the superscript $j$ denotes that the $j$-th flip $g_j$ has been applied.

For conciseness it is helpful to write the sign-flipped effective score statistic in this matrix notation. To do so note that the effective score is the product of an $1 \times n$ vector $X' {W}^{1/2} (I-{P}){V}^{-1/2}$ and an $n\times 1$ vector $(Y-\hat{\mu})$. For each $1 \leq j \leq w$, letting $G_j$ be the diagonal $n \times n$ matrix with diagonal entries $g_{j1},\dots,g_{jn}$, the sign-flipped test statistic can thus be written as
\begin{equation*}
S^{*,j}_{\hat{\fgamma}}=n^{-1/2}X'{W}^{1/2}(I-{P}){V}^{-1/2}G_j(Y-\hat{\mu}).
\end{equation*}

The test based on sign-flipping effective scores is asymptotically exact, as the following theorem states. This result coincides with Theorem 2 in \citet{hemerik2020robust}.

\begin{theorem}[{Hemerik et al., 2020}] \label{signflipeff}
For every $1\leq j \leq w$, consider the statistic $T_j^n=S^{*j}_{\hat{\fgamma}} $ and let $T^n_{(1)}\leq ...\leq T^n_{(w)}$ be the sorted test-statistics. Consider the test that rejects if $T_1^n > T^n_{\lceil(1-\alpha)w\rceil}$. 
As $n\rightarrow \infty$,   under $H_0$ the rejection probability   converges to 
$\lfloor \alpha w \rfloor /w \leq \alpha$.
\end{theorem}

We will now recall the definition of the test proposed in \citet{de2022inference}, which is a recent upgrade of the test above. 
This recent adaptation typically improves the small-sample performance of the test by standardizing the effective score. We call the resulting test statistic the standardized score statistic, defined as
\begin{equation}\label{eq_standardized}
S_{\hat{\fgamma}} = S^*_{\hat{\fgamma}}/Var\{S^*_{\hat{\fgamma}}\}^{1/2},
\end{equation}
where
\begin{equation*} \label{eq_variance}
Var\{S^*_{\hat{\fgamma}}\} = n^{-1}X^T W^{1/2} (I-P)W^{1/2}X.
\end{equation*}
The corresponding sign-flipped test statistic, for a generic flip matrix $G_j$, is defined as
\begin{equation}\label{eq_flipstandard}
S^{j}_{\hat{\fgamma}} = S^{*,j}_{\hat{\fgamma}}/Var\{S^{*,j}_{\hat{\fgamma}}\}^{1/2},
\end{equation}
where 
\begin{equation*} \label{eq_flipvariance}
Var\{S^{*,j}_{\hat{\fgamma}}\} = n^{-1}X^T W^{1/2} (I-P)G_j(I-P)G_j (I-P)W^{1/2}X.
\end{equation*}

In practice, this variance has to be estimated, replacing $W$ with an estimate $\hat{W}=\hat{D}\hat{V}^{-1}\hat{D}$, where $\hat{D},\hat{V}$ are estimates of $D,V$. This introduces an asymptotic approximation of lower order with respect to the parameter estimation. Indeed, in a first-order analysis the use of the true or the estimated variance is equivalent \citep{pace1997principles}. All results in the rest of the paper will be formulated in terms of $W$, but the results will hold up to the first order in terms of $\hat{W}$ \citep{de2022inference}.

The test based on sign-flipping standardized scores is asymptotically exact, as the following theorem states (for a discussion about the faster convergence, with respect to the effective scores, see \cite{de2022inference}). This result coincides with Proposition 1 in \citet{de2022inference}.

\begin{theorem}[{De Santis et al., 2022}] \label{signflipstd}
For every $1\leq j \leq w$, consider the statistic $T_j^n=S^{j}_{\hat{\fgamma}} $ and let $T^n_{(1)}\leq ...\leq T^n_{(w)}$ be the sorted statistics. Consider the test that rejects if $T_1^n > T^n_{\lceil(1-\alpha)w\rceil}$. 
As $n\rightarrow \infty$,   under $H_0$ the rejection probability   converges to 
$\lfloor \alpha w \rfloor /w \leq \alpha$.
\end{theorem}

Some robustness properties of this test are discussed in \citet{de2022inference}. In particular, the proposed test is proven to be asymptotically exact in case of general variance misspecification under minimal assumptions.

\section{Multiple dependent variables} \label{sec:multiv}
We now generalize the results of Section \ref{Sect:univartest} to the case of multiple generalized linear models fitted in parallel.
Suppose there are $m\geq 2$ dependent variables $Y^1,\dots,Y^m$ and that for each response $Y^l$ we have $n$ independent observations $Y^l_1,\dots,Y^l_n$, which follow some model in the exponential dispersion family. 
We consider $m$ null hypotheses $H_1,\dots,H_m$, where $H_l$ is the hypothesis that $\beta^l=\beta^l_0$.
For the $l$-th model we assume the same assumptions of Section \ref{Sect:univartest}.

This section introduces the multiple testing procedure. It involves computing the effective score for each of the dependent variable $Y^1,\dots,Y^m$. The effective score for the $l$-th response is then 
\begin{equation*}
    S^{*,l}_{\hat{\fgamma}}=n^{-1/2}X' ({W}^l)^{1/2} (I-{P}^l)({V}^l)^{-1/2}(Y^l-\hat{\mu}^l)=n^{-1/2}\sum_{i=1}^{n}\nu^{*l}_{\hat{\fgamma},i},
\end{equation*}
where the superscript $l$ indicates that the corresponding quantities have been for the $l$th model as in Section \ref{Sect:univartest}. Analogously to Section \ref{Sect:univartest}, we define the sign-flipped effective score test statistic for a generic flip matrix $G_j$ as
\begin{equation*}
S^{*,j,l}_{\hat{\fgamma}}=n^{-1/2}X' ({W}^l)^{1/2} (I-{P}^l)({V}^l)^{-1/2}G_j(Y^l-\hat{\mu}^l)=n^{-1/2}\sum_{i=1}^{n}g_{ji}\nu^{*l}_{\hat{\fgamma},i}.
\end{equation*}
Using the same sign-flip for each model allows the procedure to preserve the dependence structure. Moreover, we have
\begin{equation*}
    Var\left\{S^{*,j,l}_{\hat{\fgamma}}\right\}=X'({W}^l)^{1/2}(I-{P^l})G_j(I-{P^l})G_j(I-{P^l})({W}^l)^{1/2}X.
\end{equation*}
This allows us to define the standardized score as
\begin{equation*}
    S^{j,l}_{\hat{\fgamma}}= S^{*,j,l}_{\hat{\fgamma}}Var\left\{S^{*,j,l}_{\hat{\fgamma}}\right\}^{-1/2}.
\end{equation*}
We can build a multivariate test statistic which takes into account the standardization of the joint variance-covariance matrix as follows. Let
\begin{equation}\label{eq:globalnull}
    H_L= \bigcap_{l \in L} H_l: \beta^l=\beta^l_0
\end{equation}
where $L$ is a fixed set containing any combination of the $m$ hypotheses.
Note that $H_L$ can be either a global or a partial null hypothesis for a subset of parameters. For simplicity, we will derive the procedure for the global null hypothesis, but it is exactly analogous for partial null hypotheses of any kind. 

A first idea to perform the test might be given by the following full standardization approach which builds a global test statistic based on the form of a Mahalanobis distance as follows.
Let $S^{*,j,L}_{\hat{\fgamma}}$ be the $m$-dimensional vector with elements $S^{*,j,l}_{\hat{\fgamma}}$, $1 \le l \le m$. Let $Var\{S^{*,j,L}_{\hat{\fgamma}}\}$ denote the corresponding variance-covariance matrix, which has dimension $m \times m$ and is assumed (for the moment) known. Let the joint test statistic be
\begin{equation}\label{eq:full_standardized}
    T_{j,L}^n=\left(S^{*,j,L}_{\hat{\fgamma}}\right)' Var\left\{S^{*,j,L}_{\hat{\fgamma}}\right\}^{-1} \left(S^{*,j,L}_{\hat{\fgamma}}\right).
\end{equation}
for $1\le j\le w$. It can be proven that an asymptotic $\alpha$-level test can be built with a reasoning analogous to Theorem \ref{signflipstd}.

However there are some issues with this approach. In particular for each flip it is necessary to invert an $m \times m$ matrix, which becomes infeasible for large values of $m$. Moreover, it requires estimating the correlation between responses, which must be assumed to be known except for a limited number of parameters; for instance, we might choose to assume the correlation of the responses to be equal between different observations. Finally, in order to perform an overall analysis with post-hoc validity, we might choose a closed testing approach \citep{marcus1976closed}, which requires to perform all the $2^m$ possible tests, which can be very demanding for growing $m$.

A fast alternative, which is more appealing for large values of $m$, consists of doing marginal standardization of the test statistic. We will start from the effective scores and then proceed to the standardized scores showing that we are able to obtain asymptotically valid inference for both procedures.

Let $\bm{M}^n$ be the $w$-by-$m$ matrix with $(j,l)$-th entry equal to $S^{*j,l}_{\hat{\fgamma}}$. 
The following lemma will be fundamental in proving that the proposed multiple testing methods are asymptotically exact. 



\begin{lemma} \label{lemmascorematrix}
Let $\bm{M}^n$ as defined above. Then, for $n\rightarrow\infty$, $\bm{M}^n$ converges in distribution to $\bm{M}$, where  all rows of $\bm{M}$ have the same multivariate normal distribution.
\end{lemma}

\begin{proof}
Let $\bm{M}^n_{0}$ be the $w$-by-$m$ matrix with $(j,l)$-th entry equal to $S^{*,j,l}_{\fgamma_0}$. Note that $\bm{M}^n_{0}$ is based on knowledge of the true nuisance parameters $\fgamma_0$, where $S^{*,j,l}_{\fgamma_0}$ is the corresponding effective score. The consequence is that each entry of the matrix $\bm{M}^n_{0}$ is the sum of $n$ independent (flipped) score contributions \citep{hemerik2020robust}. 
In particular each row of $\bm{M}^n_{0}$ is uncorrelated with the other rows, due to the independence of the flips.
Moreover, the correlation structure within each row coincides with the correlation structure of the contributions 
$\nu^{*1}_{\fgamma_0},\dots, \nu^{*m}_{\fgamma_0}$.
This follows as, the sum of $n$ independent vectors with the same correlation structure has the same correlation structure as the summands, and dividing by $n^{1/2}$ leaves the correlation structure unchanged.
Consequently, the multivariate central limit theorem \citep{van1998asymptotic} implies that $\bm{M}^n_{0}$ converges in distribution to some matrix $\bm{M}_{0}$, which has identically distributed multivariate normal rows.

Now it remains to show that $\bm{M}$, i.e., the matrix based on the \emph{estimated} nuisance parameters, also has identically distributed multivariate normal rows.
As shown in the proof of Theorem 2 in \citet{hemerik2020robust}, we have 
$S^{*,j,l}_{\hat{\fgamma}}= S^{*,j,l}_{\fgamma_0}+ o_{p}(1)$, $1\leq j \leq w$, $1\leq l \leq m$.
 This means that $\bm{M}$ is asymptotically equivalent to $\bm{M}_{0}$. Hence the result holds.
\end{proof}

If we used standardized scores, instead of effective scores, to fill the matrix $\bm{M}^n$, then Lemma \ref{lemmascorematrix} still holds, since the standardized scores are asymptotically equivalent to the effective (unstandardized) scores, see \citep{de2022inference}.

From Lemma \ref{lemmascorematrix} we can build an asymptotic exact $\alpha$-level test for any composite hypothesis as follows. Let $H_L$ be defined as in \eqref{eq:globalnull}. Given any non-decreasing function $\psi:\mathbb{R}^k \rightarrow \mathbb{R}$, for each $1 \leq j \leq w$, define global flipped test-statistics
\begin{equation}\label{eq:teststat}
T^n_j=\psi\left(|S^{j,1}_{\hat{\gamma}}|,\dots, |S^{j,L}_{\hat{\gamma}}|\right).
\end{equation}
The following theorem shows that we can use these test statistics in order to obtain an asymptotic $\alpha$-level test.

\begin{theorem} \label{theorem:multi_test}
For every $1\leq j \leq w$, consider the statistic $T_j^n $ as defined in \eqref{eq:teststat} and let $T^n_{(1)}\leq ...\leq T^n_{(w)}$ be the sorted statistics. Consider the test that rejects if $T_1^n > T^n_{\lceil(1-\alpha)w\rceil}$. 
As $n\rightarrow \infty$,   under $H_L$ the rejection probability   converges to 
$\lfloor \alpha w \rfloor /w \leq \alpha$.
\end{theorem}
\begin{proof}
Lemma \ref{lemmascorematrix} implies that the test statistics $T^n_1,\dots,T^n_w$ are asymptotically independent and identically distributed.
Note that, by the definition of $\psi$, high values of $T_1$ shows evidence against the null hypothesis {$H_L$}. Hence, Lemma 1 of \cite{hemerik2020robust} directly applies and obtain anasymptotic $\alpha$-level test.
\end{proof}
Theorem \ref{theorem:multi_test} implies that we can build an asymptotic valid local test for the hypothesis \eqref{eq:globalnull}, which therefore guarantees weak control of the familywise error rate (FWER), that is, the probability of making any false rejection under the global null hypothesis. 
 


Multiple choices of the function $\psi$ are available \citep{pesarin2001multivariate} and the choice will influence the power properties in different settings.
We can subsequently apply the closed testing approach \citep{marcus1976closed} to build a procedure which guarantees strong control of the FWER by computing the $2^n$ intersection tests, and this procedure is optimal in the sense that every FWER controlling procedure is equal or can be improved by applying the closed testing principle \citep{goeman2021closed}.
The strong control of the FWER ensures control of the probability of making any false rejection for any combination of true and false hypotheses.

However, the number of tests can become unfeasible for large values of $m$. A dramatic shortcut is given by selecting the maximum of the test statistics as combining function  and is referred to as the max-$T$ method. There are roughly two versions of the max-$T$ method by Westfall and Young \citep{westfall1993resampling, westfall2008multiple,meinshausen2011asymptotic}: the single-step method and the sequential method. The single-step method is simpler and faster, while the sequential method is more powerful.
The single-step approach, based on the matrix $\bm{M}^n$ of test statistics, is defined as follows. Here we formulate a version that employs two-sided tests.
For every $1\leq j \leq w$, let $m_j$ be the maximum of the test statistics $|\bm{M}^n_{j,l}|$, $1\leq l \leq m$. Let $m_{(1)},\dots,m_{(w)}$ be the sorted values $m_1,\dots,m_w$. Then the multiple testing method  rejects all hypotheses with index $l$ for which $|\bm{M}^n|_{1,l}> m_{(\lceil (1-\alpha)w\rceil)}$.  
The sequential max-$T$ method is defined as follows; after the first step defined above is completed, the procedure is continued in a step-down way. We remove from the matrix $\bm{M}^n$ all rows corresponding to the hypotheses rejected in the first step, then we reapply the same procedure described above. This process is continued until there are no more rejections. Note that the procedure with standardized test statistic is defined in the same way.

The following theorem states that the single-step and sequential max-$T$ methods provide strong asymptotic FWER control. Write ${\rm FWER}_n$ to indicate potential dependence of the FWER on $n$.
\begin{theorem}
For both the  single-step and sequential max-$T$ method, $\limsup_{n\rightarrow\infty}({\rm FWER}_n) \leq \alpha.$
\end{theorem}

\begin{proof}
For both the single-step and the sequential max-$T$ method, the argument is as follows.
Recall that $\bm{M}^n$ converges in distribution to  $\bm{M}$. Let $\bm{M}_{\N}$ be the submatrix of $\bm{M}$ that only contains the rows corresponding to the true null hypotheses.
If the matrix $\bm{M}$ is used as input for the multiple testing procedure, then the rows of the submatrix $\bm{M}_{\N}$ are exchangeable, i.e. swapping rows does not change the distribution of the matrix. This implies that the maxT method provides strong FWER control \citep[see e.g.][Theorem 2]{goeman2010sequential}.

If instead $\bm{M}^n$ is used as input, then FWER control follows from the continuous mapping theorem \citep{van1998asymptotic}, since  $\bm{M}^n$ converges to $\bm{M}$ in distribution. 
\end{proof}

\section{Simulation Study}\label{sect:sims}

\subsection{Univariate}
\begin{figure}
    \centering
    \includegraphics[width=\linewidth]{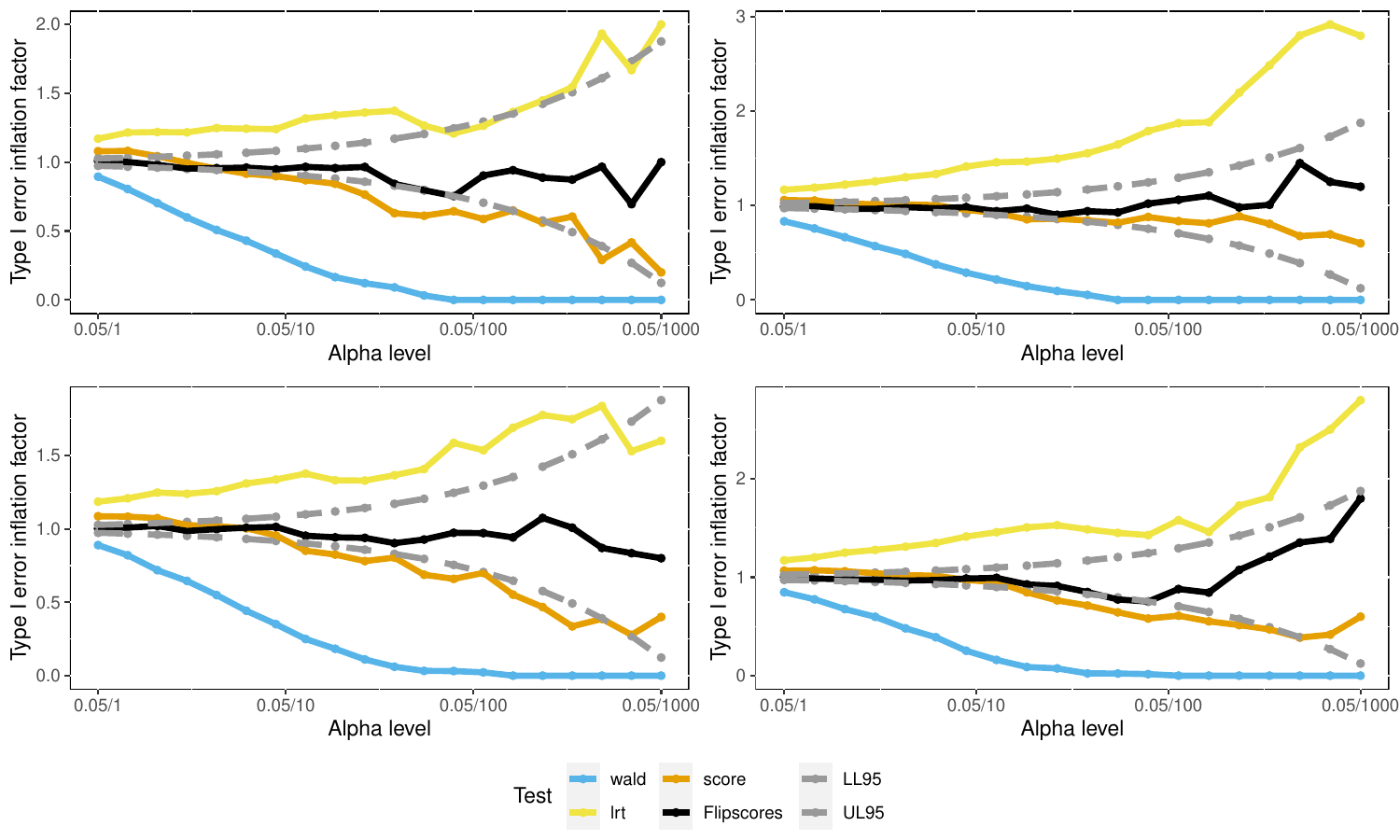}
    \caption{Logit model, univariate. On abscissa the Bonferroni-adjusted alpha for the control of FWER at level 5\%, i.e. .05/m. The flipscores is always inside che confidence bands, while the likelihood ratio test shows an anti-conservative behavior and the other two tests are conservative}
    \label{fig:sim_uni}
\end{figure}

We first show a simulation study for univariate tests in Figure \ref{fig:sim_uni}. The true model is a logistic regression model with one target covariate $X$ and one nuisance covariate $Z$, whose sample size is $n=50$. We test the target parameter $H_0:\beta=0$ setting the true regression parameter $\beta=0.$ Four settings are considered: the top-left has true $\gamma$ (the nuisance parameter) equal to 0 and correlation between $X$ and $Z$ equal to 0, the top-right has true $\gamma$ equal to 1 and correlation between $X$ and $Z$ equal to 0, the bottom-left  has true $\gamma$ equal to 0 and correlation between $X$ and $Z$ equal to 0.5 and the bottom-right has true $\gamma$ equal to 1 and correlation between $X$ and $Z$ equal to 0.5. A total of 100\,000 simulations have been run.
We compare the standardized flipscores approach with three standard parametric competitors, namely the Wald, Score and Likelihood Ratio (LRT) tests.
The flipscores test is based on 2\,000 random flips (we refer to \cite{hemerik2018exact} for a discussion about the use of a limited number of random flips). We plot 95\% simulation confidence bands for the actual control of the significance level.
The $x$ axis represents the significance level as a function of the number of hypotheses tested, while the $y$ axis represents the ratio between the Empirical type I error and the nominal level $\alpha$. A Bonferroni correction is applied to the number of hypotheses tested. Note that this plot aims to show the control of the type I error of each testing approach given the actual significance level of the overall analysis due to the number of tests involved.

We can observe that the Flipscores is generally inside the 95\% simulation confidence bands. The Wald test is very conservative, the parametric score test is slightly conservative while the likelihood ratio test does not control false positives. More simulations can be find in the appendix for the different settings with a sample size of $n=100$, where we observe similar results.

\subsection{Multivariate}
\begin{figure}
    \centering
    \includegraphics[width=\linewidth]{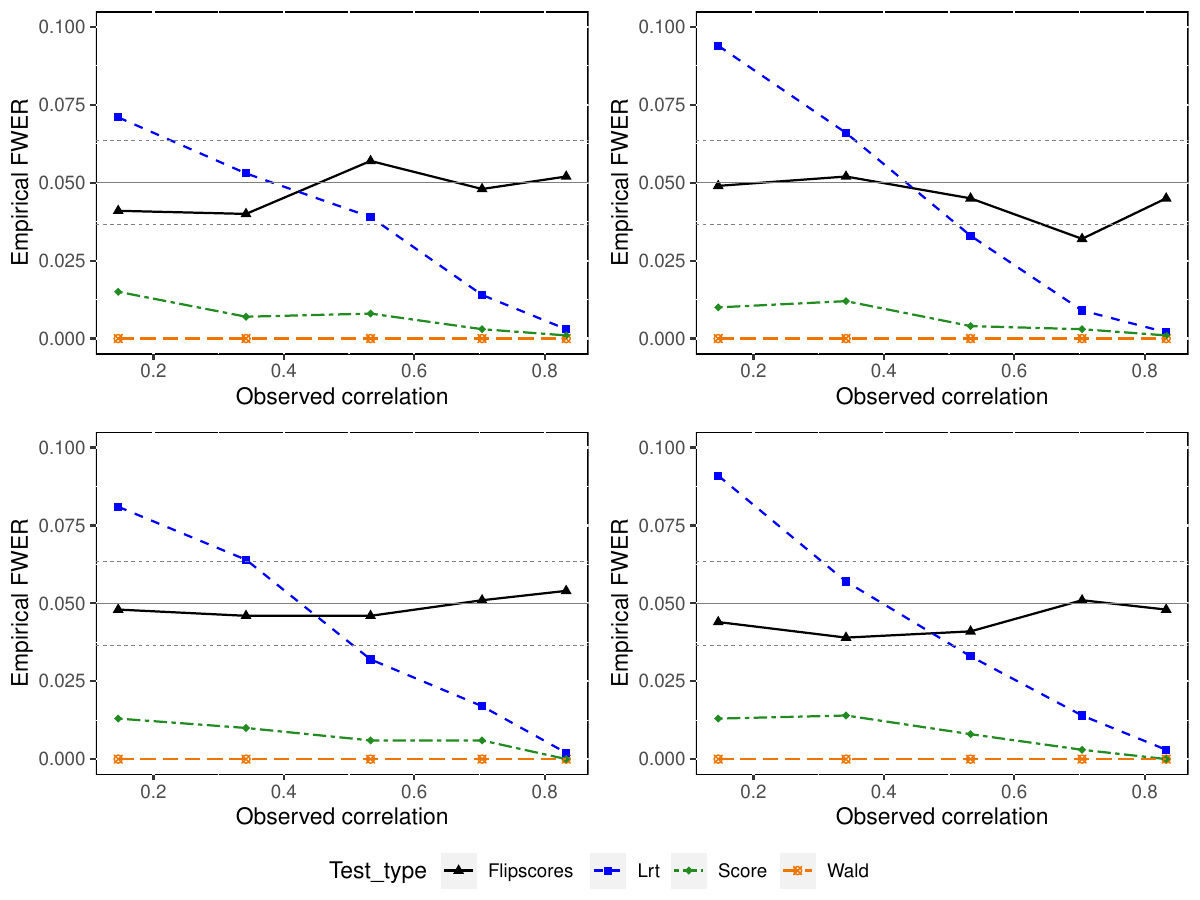}
    \caption{FWER for the multivariate logit model. The flipscores approach control the FWER in each setting. Instead the Wald and score tests are conservative whilst the likelihood ratio test is inflated at low correlations and conservative at high correlations.}
    \label{fig:sim_multi}
\end{figure}

\begin{figure}
    \centering
    \includegraphics[width=\linewidth]{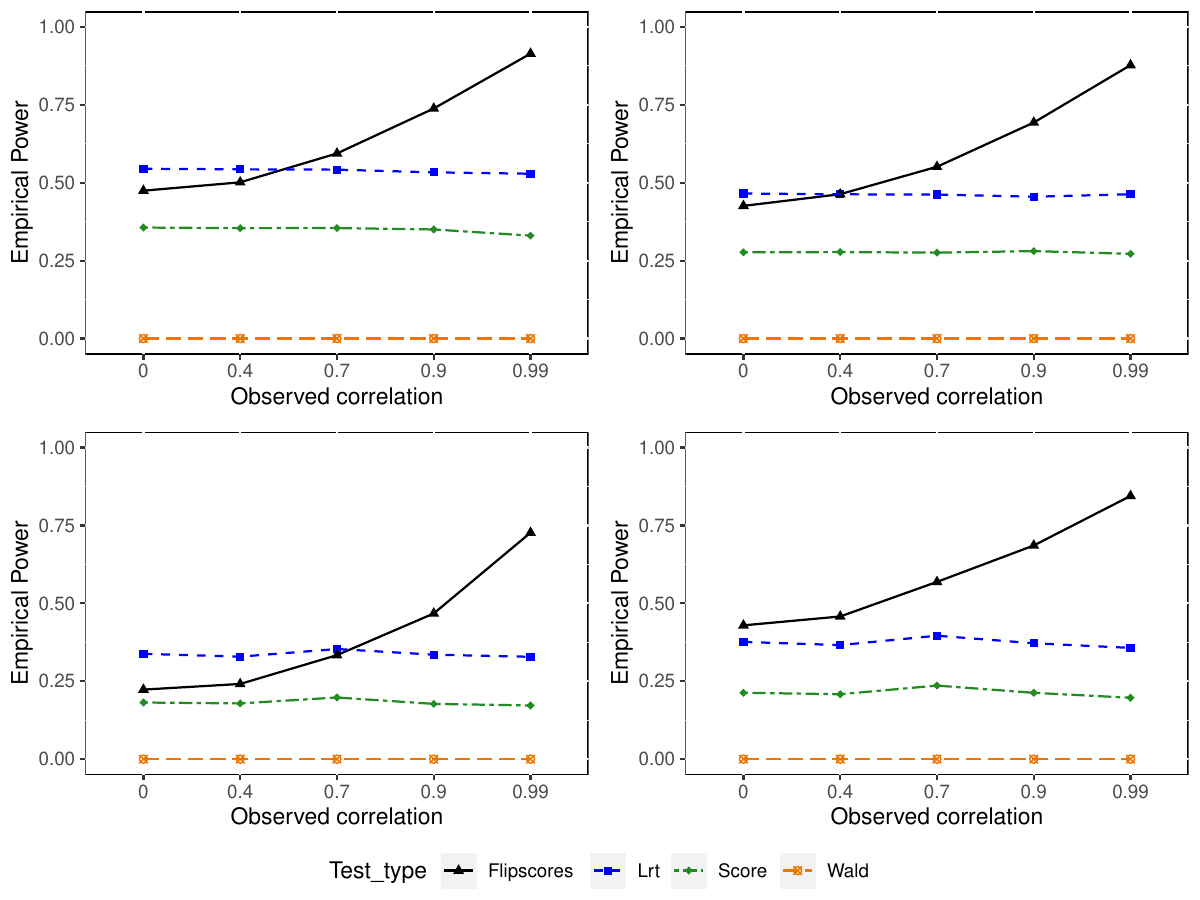}
    \caption{Power for the multivariate logit model. The Flipscores approach is more powerful with respect to the other approaches as the correlation increases.}
    \label{fig:sim_multi2}
\end{figure}

Figures \ref{fig:sim_multi} and \ref{fig:sim_multi2} represent a multivariate simulation.
We set a total of 1\,000 dependent variables. For 20\% they have $\beta=1$ and remaining 80 \% have $\beta=0$. It means that we 1\,000 logistic regression model, one for each response, which are modelled as function of one common target covariate $X$ with parameter $\beta$ and one common nuisance covariate $Z$ with parameter $\gamma$. Four settings are considered: the top-left has true $\gamma$ equal to 0 and correlation between $X$ and $Z$ equal to 0, the top-right has true $\gamma$ equal to $-1$ and correlation between $X$ and $Z$ equal to 0, the bottom-left has true $\gamma$ equal to 0 and correlation between $X$ and $Z$ equal to 0.5 and the bottom-right has true $\gamma$ equal to $-1$ and correlation between $X$ and $Z$ equal to 0.5.

For controlling the FWER we apply the proposed method with the standardized flipscores, using  2\,000 random flips, while the other parametric competitors are corrected with the Bonferroni-Holm procedure.

The $x$ axis of both figures represents the average observed correlation between the dependent variables. 
Figure \ref{fig:sim_multi} represents the empirical FWER of the true null hypothesis. The flipscores is closed to the nominal level, the LRT is anticonservative for low correlation, while the other two methods are highly conservative. Figure \ref{fig:sim_multi2} illustrates the empirical power for the true alternatives. The flipscores approach has the greatest power, especially for higher levels of correlation, as the other methods do not take advantage of the correlation structure.
More simulations can be find in the appendix for the different settings with a sample size of $n=100$, where we observe similar results.

\section{Conclusion}\label{sec:concl}

This paper builds on recent state-of-the-art developments in high-dimensional inference and semi-parametric statistics. We provide the first permutation-type approach for powerful, robust multiple testing in GLMs with many responses. This represents an important step in the development of permutation methods for complex data.

In this paper we focus on the control of the FWER and on the power properties of the multiple testing procedure. However, we can do further inference; indeed, the procedure based on the sign-flip test can be easily extended in order to estimate a lower bound for the true discovery proportion (TDP), that is, the proportion of non-null coefficients over a pre-specified group of hypotheses, using a permutation-based framework \citep{goeman2011multiple, blain2022notip, andreella2023permutation, vesely2023permutation}.

In future work, we aim to zoom in on applications of our approach to neuroimaging data and RNA-Seq data. Since GLMs are so widely applicable, we expect there will be many more applications where our approach proves to be useful.


\bibliography{references.bib}
\bibliographystyle{plainnat}

\newpage

\section{Appendix A}

\begin{figure}
    \centering
    \includegraphics[width=\linewidth]{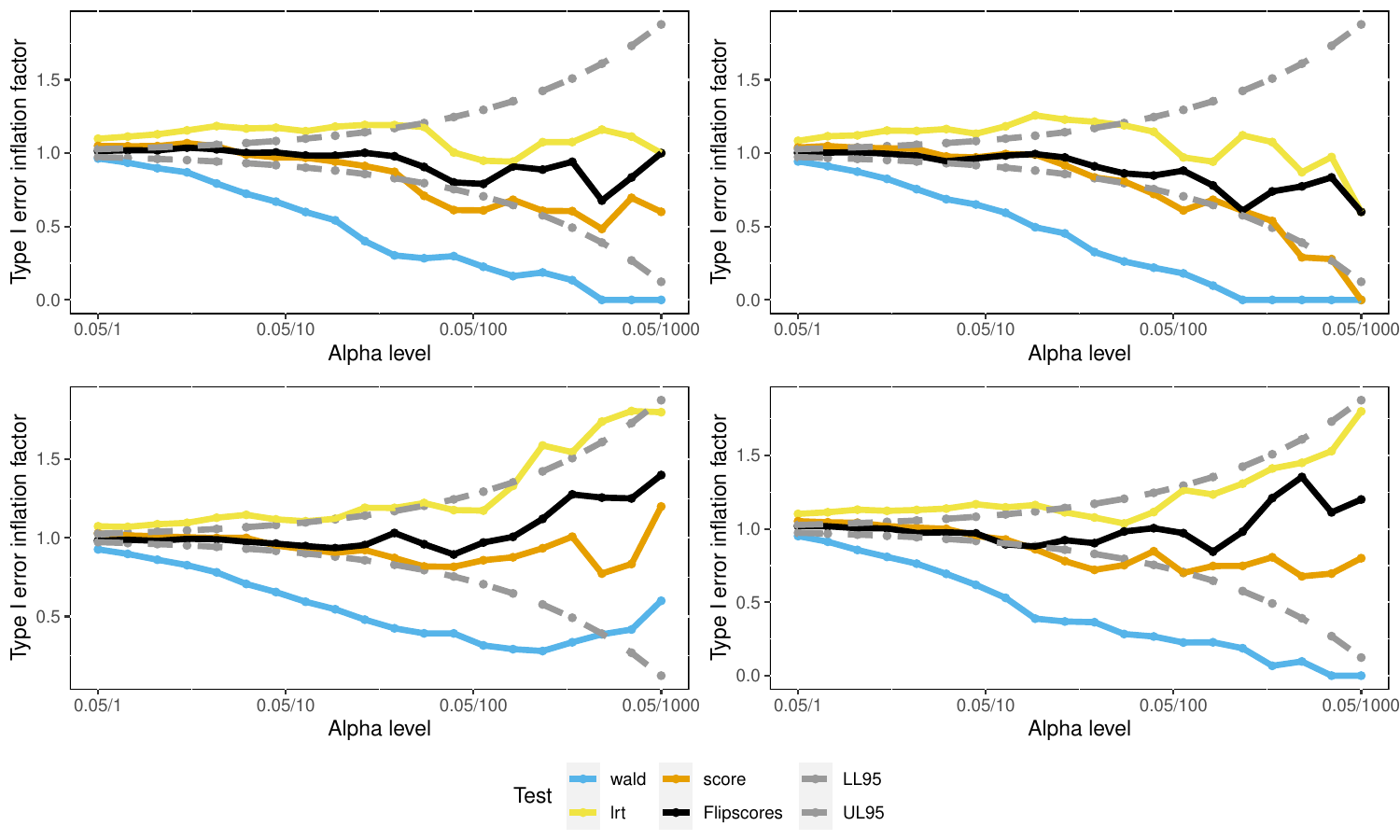}
    \caption{Logit model, univariate. The results can be interpreted similarly to those shown in Figure \ref{fig:sim_uni}, with some improvement of the parametric competitors}
    \label{fig:sim_uni2}
\end{figure}

\begin{figure}
    \centering
    \includegraphics[width=\linewidth]{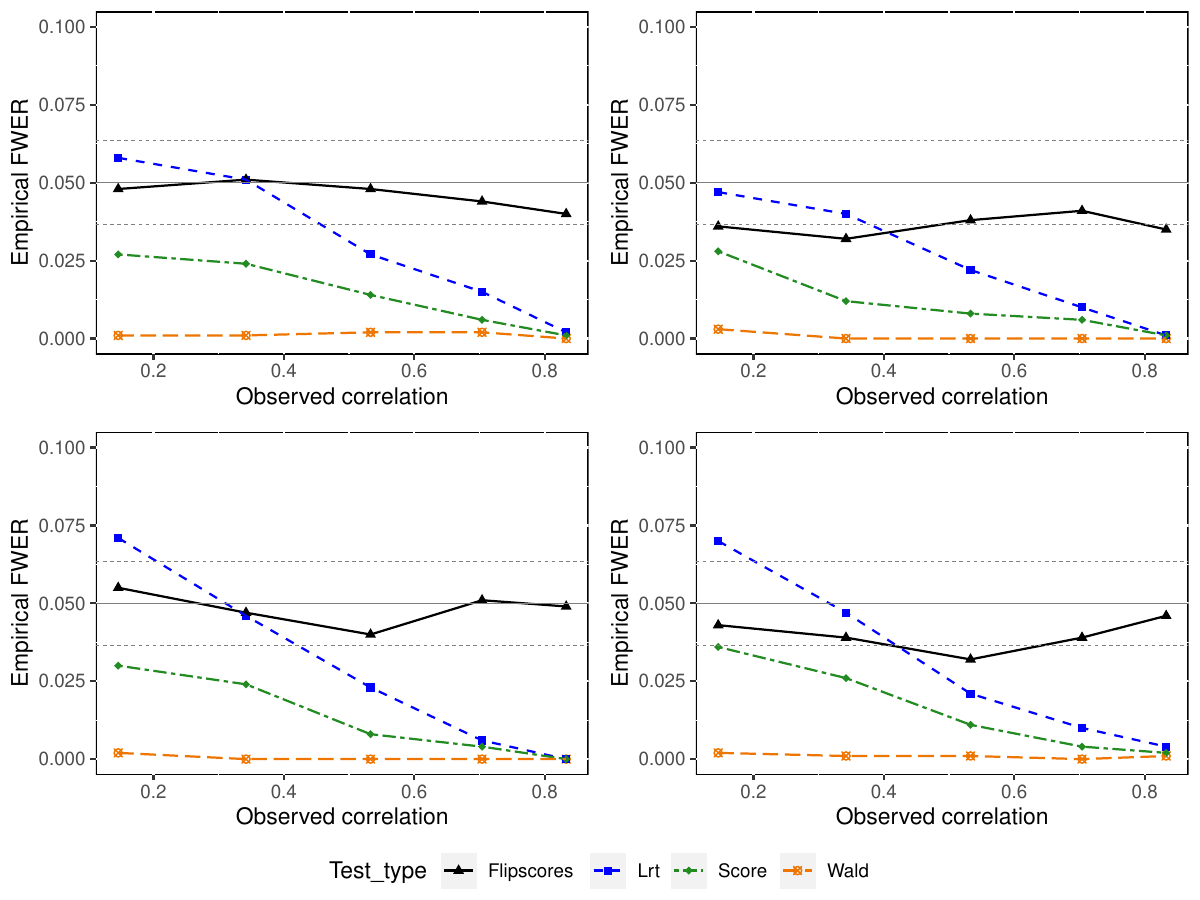}
    \caption{FWER for the multivariate logit model. The results can be interpreted similarly to those shown in Figure \ref{fig:sim_multi}}
    \label{fig:sim_multi3}
\end{figure}

\begin{figure}
    \centering
    \includegraphics[width=\linewidth]{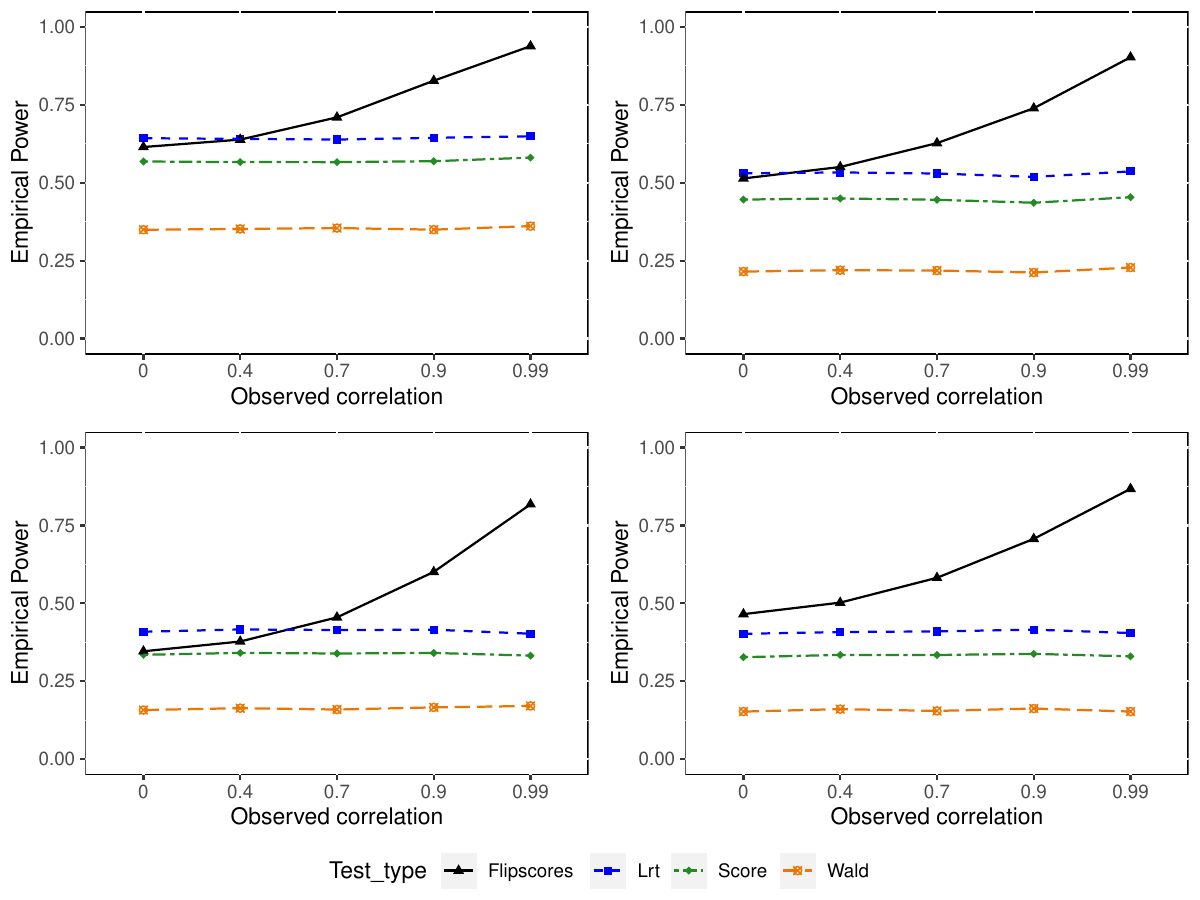}
    \caption{Power for the multivariate logit model. The results can be interpreted similarly to those shown in Figure \ref{fig:sim_multi2}}
    \label{fig:sim_multi4}
\end{figure}

\end{document}